\documentclass[a4paper,oneside,12pt]{article}
\usepackage[titletoc,toc,title]{appendix}

\usepackage[latin1]{inputenc}
\usepackage{latexsym}
\usepackage{hyperref}
\usepackage[all]{xy}
\usepackage{amsmath,amssymb,amsthm,xspace,a4wide,color}

\numberwithin{equation}{subsection}

\newcommand{\HoC}{H\mbox{o}(\cc{C})}
\newcommand{\HoCr}{H\mbox{o}(\cc{C},r)}
\newcommand{\HoCo}{H\mbox{o}(\cc{C}_0)}

\newlength\Colsep
\theoremstyle{definition}
  \begingroup
        \newtheorem{remark}[equation]{Remark}
        \newtheorem{example}[equation]{Example}
        \newtheorem{sinnadastandard}[equation]{}

	\newtheorem{notation}[equation]{Notation}
\endgroup

\theoremstyle{plain}
  \begingroup
        \newtheorem{theorem}[equation]{Theorem}
        \newtheorem{lemma}[equation]{Lemma}
        \newtheorem{proposition}[equation]{Proposition}
        \newtheorem{corollary}[equation]{Corollary}

	    \newtheorem{definition}[equation]{Definition}

\endgroup

\newcommand{\mr}[1]{\stackrel{#1}{\longrightarrow}}
\newcommand{\ml}[1]{\stackrel{#1}{\longleftarrow}}

\newcommand{\Mr}[1]{\stackrel{#1}{\Rightarrow}}

\newcommand{\xr}[1]{\xrightarrow{#1}}

\newcommand{\cc}[1]{\mathcal{#1}}
\newcommand{\C}{\mathcal{C}}

\newcommand{\eps}{\varepsilon}

\newcommand{\adj}[2]{
\ar@/^1ex/[r]^{{#1}}
\ar@{}[r]|{\bot}
\ar@/_1ex/@{<-}[r]_{{#2}} }

\newcommand{\jda}[2]{
\ar@/^1ex/[r]^{#1}
\ar@{}[r]|{\top}
\ar@/_1ex/@{<-}[r]_{#2} }

\newcommand{\adjbis}[2]{
\ar@/^2ex/[r]^{{#1}}
\ar@{}[r]|{\bot}
\ar@/_2ex/@{<-}[r]_{{#2}} }

\newcommand{\jdabis}[2]{
\ar@/^2ex/[r]^{#1}
\ar@{}[r]|{\top}
\ar@/_2ex/@{<-}[r]_{#2} }

\newcommand{\idavuelta}[2]{
\ar@/^1ex/[r]^{#1}
\ar@/_1ex/@{<-}[r]_{#2} }

\newenvironment{acknowledgements}
  {% Abstract > Acknowledgements
   \begin{abstract}}
  {\end{abstract}
   }

\begin{document}
\title{The homotopy relation in a category with weak equivalences}

\author{Martin Szyld}

\date{\vspace{-5ex}}

\maketitle

\begin{abstract}
We define a homotopy relation between arrows of a category with weak equivalences, and give a condition under which the quotient by the homotopy relation yields the homotopy category. 
In the case of the fibrant-cofibrant objects of a model category this condition holds, and we show that our notion of homotopy coincides with the classical one. 
We also 
show that Quillen's construction of the homotopy category of a model category, in which the arrows are homotopical classes of a single arrow between fibrant-cofibrant objects, can be made as well for categories with weak equivalences using this notion of homotopy.
We deduce from our work the saturation of model categories. 
The proofs of these results, which consider only the weak equivalences, become simpler (as it is usually the case) than those who involve the whole structure of a model category. 
\end{abstract}

\begin{acknowledgements}
I am very grateful to Eduardo J. Dubuc, this paper gained its impetus from initial collaboration with him.
\end{acknowledgements}

\section{Introduction}

In order to put our results into context, we found it appropriate to begin this introduction by recalling some thoughts regarding the theory of homotopical categories developed in \mbox{\cite[Part II]{DHKS}.} 
The starting point of this theory is the significant observation (\cite[\S 1]{CartaGroth}, \cite[25]{DHKS}) that {\em the weak equivalences of a model category determine the homotopy theory}. However, quoting still from \cite{DHKS}, 
``many model category arguments are a mix of arguments
which only involve weak equivalences and arguments which also involve cofibrations
and/or fibrations and as these two kinds of arguments have different flavors, the
resulting mix often looks rather mysterious". This is the motivation for developing a theory of categories with a distinguished class of arrows called weak equivalences, 
of which we only ask that it contains the identities and the usual two out of three axiom, 
a theory which allows to isolate the arguments which involve only this class of arrows.
It is in this spirit that the theory of homotopical categories is developed in \cite[VI]{DHKS}. 
The main point of this paper is that this same {\em philosophy} can also be applied to 
the notion of homotopy and to the arguments used to prove Quillen's localization result \cite[Theorem 1]{Quillen}.

For any model category $\C$, denote by $\HoC$ the localization of $C$ with respect to the weak equivalences, and by $\C \mr{\gamma} \HoC$ the localization functor. 
Quillen defines in \mbox{\cite[I.1,Def.4]{Quillen}} the notions or right and left homotopy between arrows of $\C$ by generalizing to this context the usual notion of homotopy, that is by using cylinder and path objects which involve the full structure of the model category, and not just the weak equivalences.
He shows that, when restricted to the fibrant-cofibrant objects, both relations coincide and are in fact a congruence (that is, an equivalence relation between arrows, stable under composition).
He finally constructs, for each object $X$ of $\C$ a 
fibrant-cofibrant object $RQX$ and his localization result states that the arrows $X \mr{} Y$ of $\HoC$ correspond to the classes of arrows $RQX \mr{} RQY$ under the homotopy relation. The arguments in his proof are a clear example of the {\em mix} mentioned in the previous paragraph (which doesn't overshadow its brilliance). In \cite[\S 10]{DHKS}, a presentation of this result is given which depends on the following two independent results:

\begin{sinnadastandard}[{\cite[10.4]{DHKS}}]
\label{sin:indep1}
The inclusion $\C_{cf} \subset \C$ of the full subcategory of fibrant-cofibrant objects induces an equivalence of categories $H\mbox{o}(\C_{cf}) \mr{} \HoC$.
\end{sinnadastandard}

\begin{sinnadastandard}[{\cite[10.6]{DHKS}}] \label{sin:indep2}
The homotopy relation in $\C_{cf}$ is the relation:
$f \sim g$ if and only if $\gamma f = \gamma g$ in $H\mbox{o}(\C_{cf})$, and furthermore 
the induced functor 
$\C_{cf}/\!\sim\ \mr{} H\mbox{o}(\C_{cf})$ is an isomorphism of categories. 
\end{sinnadastandard}

In \cite[10.3]{DHKS} a proof of \ref{sin:indep1} is given which refers only to the weak equivalences, by using the notion of deformation retract. 
However a proof of \ref{sin:indep2} is not given in \cite{DHKS}, furthermore it is described as ``long and technical" and references to ``good versions" which appear in four texts on model categories are provided instead.

\smallskip

{\em 
In this paper, we will give a notion of homotopy with respect to a family of weak equivalences, and a very simple proof of the corresponding version of \ref{sin:indep2}.}

\smallskip

The paper is organized as follows. In Section \ref{sec:prelim} we recall the basic definitions of the theory of homotopical categories, and interpret the notion of localization and the notion of quotient of a category by a congruence between its arrows as the value on objects of two functors which admit adjoints. This will allow us to study in an abstract context the problem of, given a family of arrows $\Sigma$ of a category $\C$, finding a congruence $R$ between the arrows of $\C$ such that its quotient is the localization with respect to $\Sigma$. We do this in \S \ref{sub:31}:

\smallskip 

\noindent -We find equivalent conditions on $\Sigma$ such that the desired $R$ exists, and note that 
there is always a unique possible $R$. We call this $R$ the relation of homotopy with respect to $\Sigma$, and note that one of the found conditions is the statement of a {\em Whitehead theorem} in this context. 

\smallskip 

\noindent -We say that an arrow {\em splits} if it is either a retraction or a section, and that a category with weak equivalences is {\em split-generated} if any weak equivalence is a composition of weak equivalences that split. The example in mind consists of the weak equivalences between fibrant-cofibrant objects in a model category, which as is well-known can be factored as a section followed by a retraction. We show that any split-generated category with weak equivalences satisfies the Whitehead condition mentioned above, from which our version of \ref{sin:indep2} follows. 

\smallskip

Note that by saying ``our version of \ref{sin:indep2}" above we mean that the homotopy relation considered is the relation of homotopy with respect to the weak equivalences introduced in this paper which, though {\em a posteriori} (that is assuming Quillen's localization result) must coincide with the usual notion of homotopy, it is {\em a priori} defined in a rather abstract way. Up to this point, one may say that what we have done is giving an answer (in terms only of the weak equivalences) to the question of {\em why} the homotopy category $H\mbox{o}(\C_{cf})$ is a quotient by a (uniquely determined) congruence: this is so because the family of weak equivalences in $\C_{cf}$ is split-generated.

\smallskip 

\noindent -What we do next is, still in the split-generated case, to give in \S \ref{sub:construction} a concrete description of the homotopy relation. 
It turns out to be the transitive closure of a relation which consists of considering, in Quillen's definition of homotopic arrows, an arbitrary arrow instead of an identity. 
We consider in \S \ref{sub:fork} a condition that is known to hold in the model category case under which this description is simpler, and we finally show in \S \ref{sub:modelcat}, without using Quillen's localization result, that for fibrant-cofibrant objects of a model category case this homotopy relation coincides with the classical one. 

\smallskip 

\noindent
-We say that $\C$ is deformable into $\C_0$ if there is a finite sequence of deformation retracts \mbox{$\C_0 \subset \C_1 ... \subset \C$.} We show in \S \ref{sub:replacement} that in this case we can construct the localization of $\C$ in terms of the homotopy category of $\C_0$.
Combining all our results, it follows a theorem for categories with weak equivalences, which yields Quillen's constructions and results mentioned above when considered for the case of model categories.

\smallskip

\noindent 
-We show finally that from our work it follows that, under some conditions, homotopical categories are saturated.
Since these conditions are known to hold in the model category case, a proof of the relevant result that model categories are saturated is obtained which is completely different to the one in \cite[I,5,Prop.1]{Quillen}. 
As far as we know, our conditions are also  independent to the 3-arrow-calculus axiom which is used in \cite[11.3]{DHKS} to prove saturation.

\section{Preliminaries} \label{sec:prelim}

\subsection{Categories with weak equivalences}\label{sub:catwe}

We recall from \cite{DHKS} various definitions and constructions.

\begin{definition} \label{def:catwithweyhomotcat}
Let $\C$ be a category. 
 We consider the following axioms on a family $\cc{W}$ of arrows of $\C$.
 
\smallskip

\noindent i) $\cc{W}$ contains all the identities.

\smallskip

\noindent ii) $\cc{W}$ has the {\em two out of three property}: for every pair of composable arrows $f,g$, when two of the three arrows $f,g,gf$ are in $\cc{W}$ then so is the third one.

\smallskip

\noindent iii) $\cc{W}$ has the {\em weak invertibility property}: for any arrow $f$ for which there exist arrows $g,h$ such that $fg$ and $hf$ are in $\cc{W}$, we have $f \in \cc{W}$.

A {\em category with weak equivalences} is a pair $(\C,\cc{W})$ satisfying axioms i) and ii). It is a {\em homotopical category} if it also satisfies iii).
\end{definition} 

\begin{definition} \label{def:homotfunctor}
A {\em homotopical functor} between categories with weak equivalences is a functor between the underlying categories which maps the weak equivalences to weak equivalences.

Also, given a category $\cc{D}$, a category with weak equivalences $(\C,\cc{W})$ and functors \mbox{$\cc{D} \xr{F,G} \C$,} a {\em natural weak equivalence} between $F$ and $G$ is a natural transformation $F \Mr{\theta} G$ such that $\theta_D \in \cc{W}$ for each $D \in \cc{D}$.
\end{definition}

\begin{definition} \label{def:deformation}
Given a category with weak equivalences $(\C,\cc{W})$ and a subcategory $\C_0$ of $\C$, a left (resp. right) deformation of $\C$ into $\C_0$ is a pair $(r,\theta)$ for which $\C \mr{r} \C$ is a homotopical functor satisfying $rC \in \C_0$ for each $C \in \C$, and $r \Mr{\theta} id_C$ (resp $id_C \Mr{\theta} r$) is a natural weak equivalence.
\end{definition}

\begin{definition} \label{def:construccionzigzags}
Given a category with weak equivalences $(\C,\cc{W})$, its homotopy category $Ho(\C,\cc{W})$, which we may write $\HoC$, is constructed as follows (for size issues, see for example \cite[\S 32]{DHKS}).
 $\HoC$ has the same objects of $\C$, and its arrows can be constructed by identifying {\em zigzags} of arrows (in which the backwards arrows are weak equivalences) 
when one can be obtained from the other 
by applying the following three operations, and their inverses, a finite number of times:

\smallskip

\noindent i) Omit an identity

\smallskip

\noindent ii) Compose two maps which go in the same direction

\smallskip

\noindent iii) Omit a weak equivalence which appears in both directions (with no other arrow in-between).

$\HoC$ comes equipped with a functor $\C \mr{\gamma} \HoC$ mapping each arrow to the class of its induced zigzag of length 1. 
This functor is the localization of $\C$ with respect to $\cc{W}$, in the sense that it is universal among those that map the arrows of $\cc{W}$ to isomorphisms.
\end{definition}

\subsection{The localization and the quotient adjunctions}

\begin{sinnadastandard} \label{sin:localizationadjunction}
{\bf The localization adjunction.}
Consider an arbitrary family $\Sigma$ of arrows of a category $\cc{C}$, on which we don't assume any conditions right now. We denote the localization functor in this case by $\cc{C} \mr{P_{\Sigma}} \cc{C}[\Sigma^{-1}]$. 
If we now consider the poset $\cc{A}_{\cc{C}}$ of families of arrows of $\cc{C}$ as a category, by the universal property of the localization we have a functor $\cc{A}_{\cc{C}} \mr{P} C \downarrow  \mathbb{C}at$ into the comma category of categories \emph{under} $C$. We also have a functor $C \downarrow  \mathbb{C}at \mr{I} \cc{A}_{\cc{C}}$ which maps a functor $F$ to the family of arrows $\{f \ | \ F(f) \mbox{ is invertible} \}$, and the universal property of the localization states precisely the adjunction $P \dashv I$.
\end{sinnadastandard}

\begin{sinnadastandard} 
 \label{sin:quotientadjunction}
{\bf The quotient adjunction.} We recall (\cite[II.8]{McL}) the construction of the quotient of a category $\cc{C}$ by a precongruence $R$. By a precongruence $R$ we mean for each pair of objects $A,B$ binary relation $R_{A,B}$ on $\cc{C}(A,B)$, which we will denote also by $R$ if there is no risk of confusion. Note that precongruences are a poset, with componentwise inclusion, and thus form a category $\cc{P}_{\cc{C}}$. 
A precongruence $R$ is a congruence if it satisfies simultaneously:

\smallskip

\noindent 1. Each $R_{A,B}$ is an equivalence relation 

\smallskip

\noindent 2. $R$ is closed by composition: given $f R g$ it holds $vfu \!\ R \!\ vgu$ for any arrows $u,v$ such that the compositions can be made.

The quotient of $\cc{C}$ by $R$ is a functor $\cc{C} \mr{Q_R} \cc{C}/R$ which universally identifies related arrows of $\cc{C}$, and 
$\cc{C}/R$ is constructed by taking the quotient of the sets $\cc{C}(A,B)$ by $R'$, the least congruence which contains $R$, 
see \cite[II.8]{McL} for details. We only note, since we will use this fact later, that $R'$ can be constructed as follows: first we close $R$ by composition, by defining $R^c$: 
\begin{equation}\label{eq:Rc}
f \ \! R^c \ \! g \hbox{ if and only if there exist } u,v,f',g' \hbox{ such that } f=vf'u, \, g=vg'u, \, f' R g'.
\end{equation}
Then, for each $A,B$, $R'_{A,B}$ is the least equivalence relation that contains $R^c_{A,B}$. It is an easy exercise that $R'$ is in fact a congruence.

It is immediate from the universal property of the quotient that this construction can be extended to a functor $\cc{P}_{\cc{C}} \mr{Q} C \downarrow  \mathbb{C}at$. 
Now, given any object $\cc{C} \mr{F} \cc{D}$ of $\cc{C} \downarrow  \mathbb{C}at$, we have the congruence $K_F$ given by its kernel pair, which relates two arrows of $\cc{C}$ if and only if they are mapped to the same arrow by $F$. 
It is also immediate to extend this construction to a functor $C \downarrow  \mathbb{C}at \mr{K} \cc{P}_{\cc{C}}$, and the universal property of the quotient $Q_R$ states precisely the adjunction $Q \dashv K$.

The unit $\eta_R$ of the adjunction is the inclusion $R \subseteq K_{Q_R}$. 
Note that, by the construction of $Q_R$, we have that $R$ is a congruence if and only if $R = K_{Q_R}$, i.e. if $\eta_R$ is an isomorphism in $\cc{P}_{\cc{C}}$.
It is immediate to show from these facts that $K_{Q_R}$ is the least congruence which contains $R$. 
Also note that the counit $\eps_F$ of the adjunction is given by the image factorization $\cc{C}/K_F \mr{} \cc{D}$  of $\cc{C} \mr{F} \cc{D}$, and $F$ is a strict epimorpfism if and only if $\eps_F$ is an isomorphism.
\end{sinnadastandard}

\section{Homotopy in a category with weak equivalences} \label{sec:main}

\subsection{The homotopy relation and the Whitehead condition}
\label{sub:31}

\begin{sinnadastandard} \label{sin:situation}
Let $(\C,\cc{W})$ be a category with weak equivalences, and $R$ a precongruence in $\cc{C}$. Motivated by \ref{sin:indep2}, we want to study the problem of giving conditions such that $\cc{C}/R$ is isomorphic to $Ho(\C)$ (as objects of $C \downarrow  \mathbb{C}at$). 
Also, we want to study the problem of the existence of a precongruence $R$ such that the above holds. 
Note that, by the construction of $\C / R$, it suffices to consider the case in which $R$ is a congruence. 
In view of \ref{sin:localizationadjunction}, \ref{sin:quotientadjunction}, we found it convenient to deal with this in the following general situation.
\end{sinnadastandard}

\medskip

Let 
$\xymatrix{ \cc{A} \adj{L}{R} & \cc{B} \jda{G}{F} & \cc{C}}$ be a pair of adjunctions of functors.
We denote by $\eta,\eps$ the unit and counit of $L \dashv R$, and by $\eta',\eps'$ the unit and counit of $F \dashv G$. Let $A \in \cc{A}$, $C \in \cc{C}$. We consider a pair of morphisms of $\cc{B}$, $\xymatrix{ FC \idavuelta{\varphi}{\psi} & LA }$, and its corresponding morphisms of $\cc{A}$ and $\cc{C}$ via the adjunctions:

\noindent\begin{minipage}{\textwidth}
\begin{minipage}[c][6cm][c]{\dimexpr0.5\textwidth-0.5\Colsep\relax}

$$\underline{\quad LA \mr{\psi} FC \quad }$$
\vspace{-.65cm}
$$A \mr{\widetilde{\psi}} RFC$$
\vspace{-.75cm}
\begin{equation}\label{eq:correspadj}
\end{equation}
\vspace{-.8cm}
$$\widetilde{\psi}: A \xr{\eta_A} RLA \xr{R(\psi)} RFC$$
\vspace{-.55cm}
$$\psi: LA \xr{L(\widetilde{\psi})} LRFC \xr{\eps_{FC}} FC$$

\end{minipage}\hfill
\begin{minipage}[c][6cm][c]{\dimexpr0.5\textwidth-0.5\Colsep\relax}

$$\underline{\quad FC \mr{\varphi} LA \quad }$$
\vspace{-.65cm}
$$C \mr{\widetilde{\varphi}} GLA$$
\vspace{-.75cm}
\begin{equation}\label{eq:correspadj2}
\end{equation}
\vspace{-.8cm}
$$\widetilde{\varphi}: C \xr{\eta'_C} GFC \xr{G(\varphi)} GLA$$
\vspace{-.55cm}
$$\varphi: FC \xr{F(\widetilde{\varphi})} FGLA \xr{\eps'_{LA}} LA$$

\end{minipage}%
\end{minipage}

Now, the composition $\psi \varphi$ is the identity of $FC$ if and only if we have the equality $\widetilde{\psi \varphi} = \eta'_C$ between the corresponding morphisms via the adjunction $F \dashv G$, that is if and only if $G(\psi \varphi) \eta'_C = \eta'_C$. By the definitions of $\widetilde{\varphi}$ and $\psi$ above, we have that this is if and only if the following diagram on the left commutes

\begin{equation} \label{eq:diagb1b2}
\xymatrix@C=3pc{C \ar[r]^{\widetilde{\varphi}} \ar[d]_{\eta'_C} & GLA \ar[d]^{GL(\widetilde{\psi})}  \\
GFC & GLRFC \ar[l]_{G(\eps_{FC})}}
\quad \quad
\xymatrix@C=3pc{A \ar[r]^{\widetilde{\psi}} \ar[d]_{\eta_A} & RFC \ar[d]^{RF
(\widetilde{\psi})} \\
RLA & RFGLA \ar[l]_{R(\eps'_{LA})}}
\end{equation} 

Analogously, it can be seen that the composition $\varphi \psi$ is the identity if and only if the diagram above on the right commutes. We have shown:

\begin{proposition} \label{prop:abstracto1}
For 
$\xymatrix{ \cc{A} \adj{L}{R} & \cc{B} \jda{G}{F} & \cc{C}}$, $A \in \cc{A}$, $C \in \cc{C}$, 
a pair of arrows $\xymatrix{FC \idavuelta{\varphi}{\psi} & LA}$ are mutually inverse if and only if the corresponding $\widetilde{\varphi}$, $\widetilde{\psi}$ satisfy \eqref{eq:diagb1b2} above. In particular, 
$FC$ is isomorphic to $LA$ in $\cc{B}$ if and only if there exist $C \mr{\widetilde{\varphi}} GLA$, $A \mr{\widetilde{\psi}} RFC$ satisfying \eqref{eq:diagb1b2}. \qed
\end{proposition}

We now consider only $C \in \cc{C}$, and give conditions on the existence of $A$.

\begin{proposition} \label{prop:abstracto2}
For 
$\xymatrix{ \cc{A} \adj{L}{R} & \cc{B} \jda{G}{F} & \cc{C}}$, $C \in \cc{C}$, the following statements are equivalent:

\begin{enumerate}
\item There exist $A \in \cc{A}$ such that $\eta_A$ is an isomorphism and an isomorphism $LA \mr{\psi}  FC$.
\item $LRFC \mr{\eps_{FC}} FC$ is an isomorphism.
\item There exists $C \mr{\widetilde{\varphi}} GLRFC$ such that the diagrams in \eqref{eq:diagb1b2} are commutative with $A = RFC$ and $\widetilde{\psi} = id$.
\end{enumerate}
\end{proposition}

\begin{proof} 
\begin{enumerate}
\item[$1\Rightarrow 2 \!$] 
Consider an isomorphism $LA \mr{\psi} FC$.
By the formulas in \eqref{eq:correspadj}, since $\psi$ and $\eta_A$ are isomorphisms, then so are $\widetilde{\psi}$ and $\eps_{FC}$.

\item[$2\Rightarrow 3$] 
Apply Proposition \ref{prop:abstracto1} with $A =  RFC$ and $\psi = \eps_{FC}$.

\item[$1\Rightarrow 2$]
Take $A = RFC$. By Proposition \ref{prop:abstracto1}, we have the desired isomorphism $LA \mr{\psi} FC$. By the formulas in \eqref{eq:correspadj}, since $\psi$ and $\widetilde{\psi}$ are isomorphisms, then so is $\eta_A$.
\end{enumerate}
\vspace{-.5cm}
\end{proof}

\begin{remark} \label{rem:unicoA}
From the proof of $1 \Rightarrow 2$ above, it follows that if $A$ satisfies condition $1$ then $\widetilde{\psi}$ is an isomorphism between $A$ and $RFC$.
\end{remark}

We now go back to our situation in \ref{sin:situation}. We will apply our results 
to the pair of adjunctions $\xymatrix{ \cc{P}_{\cc{C}} \: \adjbis{Q}{K} & \: \cc{C}\downarrow \mathbb{C}at \: \jdabis{I}{P} & \: \cc{A}_{\cc{C}}}$.

\begin{notation} \label{not:sigmarho}
To avoid the appearance of too many subindexes, for any precongruence $R$ we will denote the family of arrows $I_{Q_{R}}$ by $\sigma R$, and for any family of arrows $\Sigma$ we will denote the congruence $K_{P_{\Sigma}}$ by $\rho \Sigma$. Also, for arrows $f,g$ of $\C$, we will write $f \sim_\Sigma g$ to denote that $f$ is related with $g$ via $\rho \Sigma$. Note that by definition we have:

\smallskip

\noindent 1.  $\sigma R = \{X \mr{f} Y \ | \ \exists \ Y \mr{g} X, \, gf \ \! R' \ \! id_X, \, fg \ \! R' \ \! id_Y \}$, where $R'$ is the least congruence which contains $R$.  

\smallskip 

\noindent 2. $f \sim_\Sigma g$ if and only if $P_\Sigma \!\ f = P_\Sigma \!\ g$.
\end{notation}

From proposition \ref{prop:abstracto1} it follows:

\begin{corollary} \label{coro:concreto1}
For any precongruence $R$ in $\cc{C}$, and any family $\Sigma$ of arrows of $\cc{C}$, 
$Q_R$ and $P_\Sigma$ are isomorphic in $\cc{C}\downarrow \mathbb{C}at$ if and only if $\Sigma \subseteq \sigma R$ and $R \subseteq \rho\Sigma$. \qed
\end{corollary}

For a category with weak equivalences $(\C,\cc{W})$, note that to state that $Q_R$ and $P_\cc{W}$ are isomorphic in $\cc{C}\downarrow \mathbb{C}at$ means that there is a commutative diagram $\vcenter{\xymatrix@R=1pc{ & \HoC \ar@<-1ex>[dd]_{\varphi} \\
\cc{C} \ar[ru]^{\gamma} \ar[rd]_{Q_R} \\
& \cc{C}/R \ar@<-1ex>[uu]_{\psi}   }}$ in which $\varphi$ and $\psi$ are mutually inverse functors induced by the universal properties of the involved constructions (cf. \ref{sin:indep2}).
From Proposition \ref{prop:abstracto2} it follows:

\begin{corollary} \label{coro:concreto2}
For any family $\Sigma$ of arrows of $\cc{C}$, the following statements are equivalent:

\begin{enumerate}
\item There exists a congruence $R$ such that $Q_R$ and $P_\Sigma$ are isomorphic in $\cc{C}\downarrow \mathbb{C}at$.

\item $P_{\Sigma}$ is a strict epimorphism.

\item $\Sigma \subseteq \sigma \rho \Sigma$. \qed
\end{enumerate}
\end{corollary}

\begin{remark} \label{rem:unicoAaplicado} 
From Remark \ref{rem:unicoA}, it follows that $\rho \Sigma$ is the unique congruence which  may satisfy condition 1. 

By the construction of $\C / R$, we have that it is also equivalent to ask in item 1 above for the existence of a precongruence $R$ (and in this case 
$\rho \Sigma$ will be the least congruence which contains $R$).
\end{remark}

\begin{remark} \label{rem:condition2}
Note that by definition the item 2 above states that the induced functor $\C / \rho_{\Sigma} \mr{} \HoC$ is an isomorphism of categories. Since by construction this functor is the identity on objects (thus surjective), and it is faithful, then condition 2 is equivalent to stating that it is full, i.e. that any zigzag of arrows of $\C$ as in Definition \ref{def:construccionzigzags} is in the same class as a zigzag of length one.
\end{remark}

Let $(\C,\cc{W})$ be a category with weak equivalences. By analogy with the classical case, we make the following definition

\begin{definition} \label{def:homotencatwe}
We 
define the relation $\rho \cc{W}$ as the {\em homotopy relation}, we denote it also by $\sim_\cc{W}$, 
and when $f \sim_\cc{W} g$ we say that they are 
{\em homotopical} arrows. By the homotopical class of an arrow, we refer to its class in $\C /  \sim_\cc{W}$.
\end{definition}

\begin{remark}\label{rem:rhoSigmasinlocalization}
Note that, for any pair of arrows $f,g$, we have that $f \sim_\cc{W} g$ if and only if, for any functor $\C \mr{F} \cc{D}$ which maps the weak equivalences to isomorphisms, we have $Ff = Fg$. This is an alternative definition of the homotopy relation $\rho\cc{W}$ which doesn't require a construction of the localization, however assuming the existence of the localization simplifies the proofs.
\end{remark}

\begin{remark} \label{rem:unicoAaplicadoaplicado}
For any congruence $\sim$ which 
induces an isomorphism of categories
\mbox{$\C /\!\sim\ \mr{} H\mbox{o}(\C)$} (cf. \ref{sin:indep2}), by Remark \ref{rem:unicoAaplicado} we have $\sim \ = \ \sim_\cc{W}$, in other words the homotopy relation is the only possible candidate for a congruence such that the homotopy category is isomorphic to the quotient by it.
\end{remark}

\begin{remark} \label{rem:whitehead}
Note that the family $\sigma \rho \cc{W}$ is by construction (see Notation \ref{not:sigmarho}) the family of {\em homotopical equivalences}, that is the family of arrows $X \mr{f} Y$ of $\C$ such that there exists $Y \mr{g} X$ with $gf \sim_\cc{W} id_X$, $fg \sim_\cc{W} id_Y$. Item 3 in the previous Corollary is thus the statement of a Whitehead condition (see for example \cite[Th.1.10]{GJ}) for $(\C,\cc{W})$ in this context: it states that every weak equivalence is a homotopical equivalence. 
\end{remark}

\begin{definition}
We say that a category with weak equivalences $(\C,\cc{W})$ satisfies the Whitehead condition, or for short that it is Whitehead, if it satisfies the equivalent conditions of Corollary \ref{coro:concreto2}.
\end{definition}

We describe now an important class of Whitehead categories with weak equivalences.

\begin{definition} \label{def:retrsect}
Let $X \mr{s} Y$, $Y \mr{r} X$ be arrows in a category. If $rs = id_X$, $s$ is called a section for $r$, and $r$ is called a retraction for $s$.
An arrow $X \mr{s} Y$ is called a section if there exists $r$ such that $s$ is a section for $r$ and dually an arrow is called a retraction if it admits a section. An arrow that is either a section or a retraction is called a split arrow.
\end{definition}

\begin{example} \label{ex:retrsect}
We recall the following axioms from Quillen's theory of model categories (see \cite[I,1,Def.1]{Quillen} for the complete definition):

\smallskip

\noindent {\bf M1.} For any solid arrow diagram $\vcenter{\xymatrix{\cdot \ar[r] \ar[d]_i & \cdot \ar[d]^p \\
\cdot \ar[r] \ar@{.>}[ru] & \cdot
}}$ where $i$ is a cofibration, $p$ is a fibration and one of them is a weak equivalence, the dotted arrow exists.

\smallskip

\noindent {\bf M2.} Any map $f$ can be factored as a composition $pi$, where $i$ is a cofibration, $p$ is a fibration, and we may choose any one of them we prefer to be a weak equivalence. Note that if $f$ is a weak equivalence, we may take both $i$ and $p$ to be so.

If $A$ is a fibrant object and $A \mr{i} B$ is a cofibration and a weak equivalence, then using axiom {\bf M1} (see \S \ref{sub:modelcat}) we have $\vcenter{\xymatrix{A \ar[r]^{id} \ar[d]_i & A \ar[d] \\
B \ar[r] \ar@{.>}[ru]^{r} & 1
}}$ and thus $i$ is a section. Dually, if $B$ is a cofibrant object and $A \mr{p} B$ is a fibration and a weak equivalence then it is a section. Using axiom {\bf M2}, it follows that any weak equivalence between fibrant-cofibrant objects can be factored as a section followed by a retraction, both of them weak equivalences. Note that this fact is used in \cite[Proof of Th. 1.10]{GJ} in order to prove Whitehead's theorem for model categories.
\end{example}

\begin{definition} 
A category with weak equivalences
$(\C,\cc{W})$
 is {\em split-generated} if any arrow of $\cc{W}$ can be written as a composition of weak equivalences that split.
\end{definition}

\begin{proposition} \label{prop:splitiswhitehad}
Any split-generated category with weak equivalences is Whitehead.
\end{proposition}

\begin{proof}
We consider it instructive to show that in this case $(\C,\cc{W})$ satisfies both the condition 2 and the condition 3 of Corollary \ref{coro:concreto2}.

To show condition 2, by Remark \ref{rem:condition2} it suffices to show that any backward arrow is equivalent, by the relation described in Definition \ref{def:construccionzigzags}, to a forward arrow. Consider a pair of arrows $X \mr{s} Y$, $Y \mr{r} X$ such that $rs = id_X$. Then we have:

\smallskip

\noindent -If $s$ is a weak equivalence, then 
$Y \ml{s} X \,\, \sim \,\, Y \ml{s} X \mr{s} Y \mr{r} X \,\, \sim \,\, Y \mr{r} X$.

\smallskip

\noindent -If $r$ is a weak equivalence, then 
$X \ml{r} Y \,\, \sim \,\, X \mr{s} Y \mr{r} X \ml{r} Y \,\, \sim \,\, X \mr{s} Y$.

Since any weak equivalence is a composition of finite such $s$ and $r$, we conclude.

\smallskip

To show condition 3, it suffices to show that weak equivalences that split are homotopical equivalences. By the two out of three property, we can consider a pair of weak equivalences $X \mr{s} Y$, $Y \mr{r} X$ such that $rs = id_X$, and it suffices to check that $sr \sim_\cc{W} id_Y$, i.e. that $\gamma(sr) = id_Y$. 
Since $srs = s$, we have $\gamma(srs) = \gamma(s)$ from which the desired equality follows because $\gamma(s)$ is an isomorphism.
\end{proof}

Let $\C_{cf}$ be the full subcategory of fibrant-cofibrant objects of a model category $\C$, we denote by the same letter $\cc{W}$ the family of weak equivalences when restricted to $\C_{cf}$.  
By the result in Example \ref{ex:retrsect}, 
$(\C_{fc},\cc{W})$ is split-generated and thus we have:

\begin{corollary} \label{modeliswhitehead}
For any model category $\C$, the category with weak equivalences $(\C_{fc},\cc{W})$ is Whitehead. \qed
\end{corollary}

Thus, from Corollary \ref{coro:concreto2} and Remark \ref{rem:unicoAaplicadoaplicado} we have (cf. \ref{sin:indep2}, note that by definition we have $f \sim_\cc{W} g$ if and only if $\gamma f = \gamma g$ in $H\mbox{o}(\C_{cf})$)

\begin{corollary} \label{coro:cfindep2}
For any model category $\C$,
$\sim_\cc{W}$ is the unique congruence in $\C_{cf}$ such that 
 the induced functor 
$\C_{cf}/\!\sim_\cc{W} \mr{} H\mbox{o}(\C_{cf})$ is an isomorphism of categories. \qed
\end{corollary}

\subsection{A construction of the homotopy relation} \label{sub:construction}

We fix throughout this subsection a split-generated category with weak equivalences $(\C,\cc{W})$. We will give a concrete description of the homotopy relation $\sim_\cc{W}$ in this case.

\begin{definition}
We define two precongruences $R_\ell$, $R_r$ in $\C$ as follows: for $A \mr{f,g} B$, 

\smallskip

\noindent-$f \ \! R_\ell \ \! g$ if and only if there exists a weak equivalence $B \mr{\sigma} C$ such that $\sigma f = \sigma g$. 

\smallskip

\noindent-$f \ \! R_r \ \! g$ if and only if there exists a weak equivalence $C \mr{s} A$ such that $f s = g s$. 

We denote by $\sim_\ell$, resp $\sim_r$, the least congruence that contains $R_\ell$, resp. $R_r$. When $f \sim_\ell g$, resp. $f \sim_r g$, we say that $f$ and $g$ are left, resp. right homotopic.% (with respect to $\Sigma$).
\end{definition}

Note that $f \ \! R_r \ \! g$ if and only if $f \ \! R_\ell \ \! g$ when considered in the opposite category with weak equivalences. We will consider many times below only the relation of left homotopic arrows, but dual ``right" statements which we omit always hold with dual proofs.

\begin{proposition} \label{prop:3congruencescoincide}
The congruences $\sim_\ell$, $\sim_r$ and $\sim_\cc{W}$ all coincide.
\end{proposition}

\begin{proof}
We show only $\sim_\ell \!\ =  \!\ \sim_\cc{W}$, by the duality explained above. 
By 
Remark \ref{rem:unicoAaplicadoaplicado} and 
Corollary \ref{coro:concreto1} (see also Notation \ref{not:sigmarho}) it suffices to show

\smallskip

\noindent 1. For each weak equivalence $A \mr{f} B$, there exists $B \mr{g} A$ such that $gf \sim_\ell id_X$, $fg \sim_\ell id_Y$.

\smallskip

\noindent 2. If $f \!\ R_\ell \!\ g$, then $f \sim_\cc{W} g$.

Item 2 is immediate, and since $(\C,\cc{W})$ is split generated we may assume that $f$ in item 1 is split, we have thus $B \mr{g} A$, which by axiom 2 out of 3 is also a weak equivalence, such that either $gf = id_X$ or $fg = id_Y$. If $gf = id_X$, since $gfg = g \!\ id_Y$ we have $fg \sim_\ell id_Y$. 
The case $fg = id_Y$ is symmetric. 
%Dually, if $fg = id_Y$, since $fgf = f id_X$ we have $gf \sim_\ell id_Y$.
\end{proof}

We can construct $\sim_\ell$ and $\sim_r$ as follows (recall \ref{sin:quotientadjunction}). Since $R_\ell$ already satisfies that for any $A' \mr{u} A \mr{f,g} B$ we have that $f \ \! R_\ell \ \! g$ implies $fu \ \! R_\ell \ \! gu$, we can construct its closure by composition $R_\ell^c$ by ``only closing by composition on the right", i.e. as the following precongruence (cf \eqref{eq:Rc}, note that the {\em names} of the arrows are modified in order to simplify the comparison with \cite[I,1,Def. 3]{Quillen}). 

\begin{definition}
$f \!\ R_\ell^c \!\ g$ (resp $f \!\ R_r ^c \!\ g$) if and only if there is a commutative diagram of the form on the left (resp. on the right)
\begin{equation} \label{eq:diagramahomotopic}
\vcenter{\xymatrix@C=3pc@R=3pc{
A \ar@<.5ex>[r]^f \ar@<-.5ex>[r]_g \ar[d]_\alpha \ar@<.5ex>[rd]^{\partial_0} \ar@<-.5ex>[rd]_{\partial_1} & B \\
C & \widetilde{A} \ar[l]_\sigma \ar[u]_h
}}
\quad \quad \quad \quad \quad \quad \quad \quad \quad
\vcenter{\xymatrix@C=3pc@R=3pc{
\widetilde{B} \ar@<.5ex>[rd]^{d_0} \ar@<-.5ex>[rd]_{d_1} & C \ar[l]_s \ar[d]^\beta \\
A \ar[u]^k \ar@<.5ex>[r]^f \ar@<-.5ex>[r]_g & B
}}
\end{equation}
in which $\sigma$ (resp. $s$) is a weak equivalence.
\end{definition}

\begin{remark} \label{rem:cfconQuillen}
Quillen's definition of homotopic arrows in \cite[I,1,Def. 3]{Quillen} consists of demanding in addition to the above that 
$C = A$, $\alpha = id_A$ (resp. $C = B$, $\beta = id_B$). We will see below that for fibrant-cofibrant objects of a model category both notions coincide.
\end{remark}

Note that $R_\ell^c$ is clearly a reflexive and symmetric relation, thus $\sim_\ell$ will be its transitive closure. This is already an explicit description of $\sim_\ell$, and thus of $\sim_\cc{W}$: 
\begin{equation} \label{eq:explicitconstruction}
f \sim_\cc{W} g \hbox{ if and only if we have a finite sequence } f = f_0 \!\ R_\ell^c \!\ f_1 \!\ R_\ell^c \, ... \, R_\ell^c \!\ f_n = g.
\end{equation}
In the model category case, however, more can be done, and Quillen shows that his homotopy relation is transitive when $A$ is cofibrant by constructing a new homotopy whose cylinder object is obtained by {\em gluing} the cylinder objects of two composable homotopies, see \cite[Lemmas 3-4]{Quillen} for details.
Unfortunately, as far as we can tell, there is no reasonable axiom we can impose on $(\C,\cc{W})$ which would allow to mimic this construction. 
However, we could find a condition on the homotopy relation, which is known to hold in the case coming from model categories, that ensures that $R_\ell^c$ is already transitive and thus it is already the congruence $\sim_\ell$. We show this in the following subsection.

\subsection{The ``common fork" condition}
\label{sub:fork}

We consider the structures which take the place of cylinder and path objects (\cite[I,1,Def.4]{Quillen}) for categories with weak equivalences.  

\begin{definition}
Let $A \in \C$. A left fork of weak equivalences with vertex $A$ is a commutative diagram of the form $\vcenter{\xymatrix{A \ar[d]_\alpha \ar@<.5ex>[rd]^{\partial_0} \ar@<-.5ex>[rd]_{\partial_1} \\
C & \widetilde{A} \ar[l]^\sigma}}$, in which $\sigma$ and $\alpha$ are weak equivalences. Let $B \in \C$. A right fork of weak equivalences with vertex $B$ is a commutative diagram of the form $\vcenter{\xymatrix{\widetilde{B} \ar@<.5ex>[rd]^{d_0} \ar@<-.5ex>[rd]_{d_1} & C \ar[l]_s \ar[d]^\beta \\
& B}}$, in which $\sigma$ and $\beta$ are weak equivalences. Given arrows $A \mr{f,g} B$, a left (resp. right) homotopy between $f$ and $g$ is a diagram as in \eqref{eq:diagramahomotopic} with a left (resp. right) fork of weak equivalences. 
\end{definition}

\begin{remark} \label{rem:commoncylinderenmodel}
Note that the unique difference between the statement ``$f \!\ R_\ell^c \!\ g$" and the statement ``there is a left homotopy between $f$ and $g$" is that the latter requires that $\alpha$ must be a weak equivalence.

Quillen's definition in \cite[I,1,Def. 4]{Quillen} consists of demanding in addition to the above that 
$C = A$, $\alpha = id_A$, and $\partial_0 + \partial_1$ is a cofibration (resp. $C = B$, $\beta = id_B$ and $(d_0,d_1)$ is a fibration). 
Using axiom {\bf M2}, he shows then (\cite[I,1,Lemma 1]{Quillen}) that left homotopic arrows admit left homotopies. The left homotopies are the ones that yield the transitivity of the homotopic relation by the procedure described in Remark \ref{rem:cfconQuillen}.

But note that the following statement can also be shown for fibrant-cofibrant objects in a model category (see \cite[II, Cor. 1.9]{GJ} for a proof): for any cylinder object, 
left homotopic arrows admit left homotopies with respect to that cylinder object. It is this condition (actually, a weaker one in which we ask for a common fork only for each two pairs of homotopic arrows) which will allow us to prove the transitivity of the relation $R_\ell^c$ in our case. We note that this procedure (together with Proposition \ref{prop:relationscoincide}) also yields a different proof of the transitivity of the homotopy relation in the model category case.
\end{remark}

\begin{definition}
We say that $(\C,\cc{W})$ satisfies the ``common fork" condition if, given arrows $A \xr{f,g,f',g'} B$ such that $f \!\ R_\ell^c \!\ g$, $f' \!\ R_\ell^c \!\ g'$ (resp. $f \!\ R_r^c \!\ g$, $f' \!\ R_r^c \!\ g'$), 
they admit two homotopies with respect to a common fork of weak equivalences. More explicitly, there exist two commutative diagrams as on the left (resp. on the right) below
\begin{equation} \label{eq:diagramahomotopiccommonfork}
\vcenter{\xymatrix@C=3pc@R=3pc{
A \ar@<.5ex>[r]^f \ar@<-.5ex>[r]_g \ar[d]_\alpha \ar@<.5ex>[rd]^{\partial_0} \ar@<-.5ex>[rd]_{\partial_1} & B \\
C & \widetilde{A} \ar[l]_\sigma \ar[u]_h
}}
\quad \quad 
\vcenter{\xymatrix@C=3pc@R=3pc{
A \ar@<.5ex>[r]^{f'} \ar@<-.5ex>[r]_{g'} \ar[d]_\alpha \ar@<.5ex>[rd]^{\partial_0} \ar@<-.5ex>[rd]_{\partial_1} & B \\
C & \widetilde{A} \ar[l]_\sigma \ar[u]_{h'}
}}
\quad \quad \quad \quad
\vcenter{\xymatrix@C=3pc@R=3pc{
\widetilde{B} \ar@<.5ex>[rd]^{d_0} \ar@<-.5ex>[rd]_{d_1} & C \ar[l]_s \ar[d]^\beta \\
A \ar[u]^k \ar@<.5ex>[r]^f \ar@<-.5ex>[r]_g & B
}}
\quad \quad
\vcenter{\xymatrix@C=3pc@R=3pc{
\widetilde{B} \ar@<.5ex>[rd]^{d_0} \ar@<-.5ex>[rd]_{d_1} & C \ar[l]_s \ar[d]^\beta \\
A \ar[u]^{k'} \ar@<.5ex>[r]^{f'} \ar@<-.5ex>[r]_{g'} & B
}}
\end{equation}
in which $\sigma$ and $\alpha$ (resp. $s$ and $\beta$) are weak equivalences.
\end{definition}

\begin{proposition}
If $(\C,\cc{W})$ satisfies the ``common fork" condition then the relation $R^c_\ell$ is transitive.
\end{proposition}

\begin{proof}
Let $f_1 \!\ R_\ell^c \!\ f_2 \!\ R_\ell^c \!\ f_3$. We apply the ``common fork" condition with the hypothesis $f_1 \!\ R_\ell^c \!\ f_2$, $f_3 \!\ R_\ell^c \!\ f_2$. We have thus the two diagrams in the left below, from which we construct the diagram on the right
$$
\vcenter{\xymatrix@C=3pc@R=3pc{
A \ar@<.5ex>[r]^{f_1} \ar@<-.5ex>[r]_{f_2} \ar[d]_\alpha \ar@<.5ex>[rd]^{\partial_0} \ar@<-.5ex>[rd]_{\partial_1} & B \\
C & \widetilde{A} \ar[l]_\sigma \ar[u]_h
}}
\quad \quad 
\vcenter{\xymatrix@C=3pc@R=3pc{
A \ar@<.5ex>[r]^{f_3} \ar@<-.5ex>[r]_{f_2} \ar[d]_\alpha \ar@<.5ex>[rd]^{\partial_0} \ar@<-.5ex>[rd]_{\partial_1} & B \\
C & \widetilde{A} \ar[l]_\sigma \ar[u]_{h'}
}}
\quad \quad \leadsto \quad \quad
\vcenter{\xymatrix@C=3pc@R=3pc{
\widetilde{A} \ar@<.5ex>[rd]^{h} \ar@<-.5ex>[rd]_{h'} & A \ar[l]_{\partial_1} \ar[d]^{f_2} \\
A \ar[u]^{\partial_0} \ar@<.5ex>[r]^{f_1} \ar@<-.5ex>[r]_{f_3} & B
}}
$$
The diagram on the right expresses, by definition, the fact $f_1 \!\ R_r^c \!\ f_3$ (note that $\partial_1$ is a weak equivalence since $\alpha$ and $\sigma$ are so). Now, an argument dual to the above, applied to the hypothesis $f_1 \!\ R_r^c \!\ f_3$, $f_3 \!\ R_r^c \!\ f_3$, yields $f_1 \!\ R_\ell^c \!\ f_3$ as desired.
\end{proof}

\begin{corollary} \label{coro:commonfork}
For any split-generated category with weak equivalences $(\C,\cc{W})$ which satisfies the ``common fork" condition, the homotopy relation $\sim_\cc{W}$ coincides with the relations $R^c_\ell$ and $R^c_r$.
\end{corollary}

\begin{remark}
Let us say that $(\C,\cc{W})$ satisfies the ``fork" condition, which is weaker than the ``common fork" condition, if given arrows $A \xr{f,g} B$ such that $f \!\ R_\ell^c \!\ g$, (resp. $f \!\ R_r^c \!\ g$), 
they admit a homotopy with respect to a fork of weak equivalences, that is a diagram as in \eqref{eq:diagramahomotopic} in which $\alpha$ (resp. $\beta$) is an equivalence. In this case, we have by the two out of three condition that if $f \!\ R_\ell^c \!\ g$ (or if $f \!\ R_r^c \!\ g$) then $f$ is a weak equivalence if and only if $g$ is so, then by \eqref{eq:explicitconstruction} this is also the case if $f \sim_{\ell} g$ (or if $f \sim_{r} g$).  
\end{remark}

\begin{definition} \label{def:saturated}
We say that $(\C,\cc{W})$ is {\em saturated} if for each arrow $f$ of $\C$ we have the implication: if $\gamma f$ is an isomorphism then $f$ is a weak equivalence. 
\end{definition}

\begin{proposition} \label{prop:forkdasaturated}
If $(\C,\cc{W})$ is a split-generated homotopical category (see Definition \ref{def:catwithweyhomotcat}) which satisfies the ``fork" condition then it is saturated.
\end{proposition}

\begin{proof}
Let $X \mr{f} Y$ admit a homotopical inverse $Y \mr{g} X$, then by Proposition \ref{prop:3congruencescoincide} we have 
$gf \sim_\ell id_X$ 
$fg \sim_\ell id_Y$, thus by the previous remark $gf$ and $fg$ are weak equivalences and finally so is $f$ by the weak invertibility property (Definition \ref{def:catwithweyhomotcat}, item iii).
\end{proof}

\begin{corollary} \label{coro:Cfcsaturated}
If $\C$ is a model category, then $(\C_{fc},\cc{W})$ is saturated. \qed
\end{corollary}

\subsection{The equivalence between the two notions of homotopy for a model category} \label{sub:modelcat}

Let $\C$ be a model category, we consider in $\C_{fc}$ the relation $\sim_\ell$, recall that it  equals $R_\ell^c$. We will show that it coincides with Quillen's notion of left homotopic arrows.

\begin{lemma} \label{lema:sigmafibration}
For arrows
$A \mr{f,g} B$ of $\C_{fc}$ such that $f \!\ R_\ell^c \!\ g$, the diagram in \eqref{eq:diagramahomotopic} can be taken with $\sigma$ a fibration.
\end{lemma}

\begin{proof}
Since $f \!\ R_\ell^c \!\ g$, we have a commutative diagram $\vcenter{\xymatrix@C=3pc@R=3pc{
A \ar@<.5ex>[r]^f \ar@<-.5ex>[r]_g \ar[d]_\alpha \ar@<.5ex>[rd]^{\partial_0} \ar@<-.5ex>[rd]_{\partial_1} & B \\
C & \widetilde{A} \ar[l]_\sigma \ar[u]_h
}}$ in which all the objects are fibrant-cofibrant. We use axiom {\bf M2} and factorize $\sigma$ as $\widetilde{A} \mr{i} D \mr{p} C$. Note that $D$ is also a fibrant-cofibrant object. Since $i$ is a section (see Example \ref{ex:retrsect}), let $r$ be its retraction. We have thus the commutative diagram 
$\vcenter{\xymatrix@C=3pc@R=1pc{
A \ar@<.5ex>[r]^f \ar@<-.5ex>[r]_g \ar[dd]_\alpha \ar@<.5ex>[rdd]^{i \partial_0} \ar@<-.5ex>[rdd]_{i \partial_1} & B \\
& \widetilde{A} \ar[u]_h \\
C & D \ar[l]_p \ar[u]_r 
}}$
in which $p$ is a fibration as desired.
\end{proof}

\begin{remark} \label{rem:lefthomotquillenobjfibcofib}
Let us denote by $\stackrel{\ell}{\sim}$ the relation of left homotopy as in \cite[I,1,Defs. 3,4]{Quillen}: $f \stackrel{\ell}{\sim} g$ if and only if there is a commutative diagram 
$\vcenter{\xymatrix@C=3pc@R=3pc{
A \ar@<.5ex>[r]^f \ar@<-.5ex>[r]_g \ar[d]_{id} \ar@<.5ex>[rd]^{\partial_0} \ar@<-.5ex>[rd]_{\partial_1} & B \\
A & A \times I \ar[l]_\sigma \ar[u]_h
}}$, in which $\sigma$ is a weak equivalence and $A \amalg A \xr{\partial_0 + \partial_1} A \times I$ is a cofibration. 
Note that, by \cite[II, Cor. 1.9]{GJ}, we may assume $\sigma$ to be a fibration, and thus 
$A \times I$ is a fibrant object if $A$ is so.
Also note that, by \cite[I,1,Lemma 2]{Quillen}, $A \times I$ is a cofibrant object if $A$ is so.
\end{remark}

\begin{proposition} \label{prop:relationscoincide}
In $\C_{fc}$, the relations $\stackrel{\ell}{\sim}$ and $\sim_\ell$ coincide.
\end{proposition}

\begin{proof}
Given arrows $f,g$ such that $f \stackrel{\ell}{\sim} g$, by Remark \ref{rem:lefthomotquillenobjfibcofib} we have $f \!\ R_\ell^c \!\ g$. To show the other implication, let $f \!\ R_\ell^c \!\ g$. We have thus a commutative diagram $\vcenter{\xymatrix@C=3pc@R=3pc{
A \ar@<.5ex>[r]^f \ar@<-.5ex>[r]_g \ar[d]_\alpha \ar@<.5ex>[rd]^{\partial_0} \ar@<-.5ex>[rd]_{\partial_1} & B \\
C & \widetilde{A} \ar[l]_\sigma \ar[u]_h
}}$, in which we may assume $\sigma$ to be a fibration and a weak equivalence by Lemma \ref{lema:sigmafibration}. 
Using axiom {\bf M2} (see Example \ref{ex:retrsect}), we construct a cylinder object for $A$, $A \amalg A \xr{\partial_0' + \partial_1'} A \times I \mr{\sigma'} A$, and using axiom {\bf M1} we have
$\vcenter{\xymatrix{
A \amalg A \ar[r]^-{\partial_0 + \partial_1} \ar[d]_{\partial_0' + \partial_1'} & \widetilde{A} \ar[d]^\sigma \\
A \times I \ar[r]_-{\alpha \sigma'} \ar@{.>}[ru]^{k} & C
}}$. The arrow $A \times I \mr{hk} B$ yields the desired homotopy.
\end{proof}

Combining  
Propositions \ref{prop:3congruencescoincide} 
 and \ref{prop:relationscoincide}, it follows that all the considered notions of homotopy coincide for fibrant-cofibrant objects and thus 
\ref{sin:indep2} follows from Corollary \ref{coro:cfindep2}.

\subsection{Replacement in a category with weak equivalences}
\label{sub:replacement}

We fix a category with weak equivalences $(\C,\cc{W})$ and a subcategory $\C_0$ of $\C$. We denote by $\cc{W}_0$ the restriction of $\cc{W}$ to $\C_0$, we consider the subcategory with weak equivalences $(\C_0,\cc{W}_0)$ and denote by $\C_0 \mr{\gamma_0} \HoCo$ its homotopy category.

\begin{definition} \label{def:pointwisedeformation}
A pointwise left (resp. right) deformation of $\C$ into $\C_0$  
 consists of giving for each object $X$ a weak equivalence $rX \mr{\theta_X} X$ (resp. $X \mr{\theta_X} rX$), with $rX \in \C_0$, and for each arrow $X \mr{f} Y$ a commutative diagram $\vcenter{\xymatrix{rX \ar[d]_{\theta_X} \ar[r]^{rf} & rY \ar[d]^{\theta_Y} \\
X \ar[r]_f & Y}}$ (resp. $\vcenter{\xymatrix{X \ar[d]_{\theta_X} \ar[r]^{f} & Y \ar[d]^{\theta_Y} \\
rX \ar[r]_{rf} & rY}}$). We denote it by $r: \C \curvearrowright \C_0$, omitting to write explicitly the arrows $\theta_X$.

\smallskip

\noindent
Note that, by the two out of three property, $rf$ is a weak equivalence if and only if $f$ is so.
\end{definition}

Recalling Definition \ref{def:deformation}, we note that a deformation of $\C$ into $\C_0$ is a pointwise deformation in which $r$ is a functor. We also denote deformations by $r: \C \curvearrowright \C_0$.

\begin{example}
The fibrant and cofibrant replacements $R$ and $Q$ in \cite[I,1, proof of Th.1]{Quillen} are examples of left and right pointwise deformations. It was observed later (see for example \cite[3.3, v)]{DHKS}) that for all the significant examples it was not a problem to consider functorial replacements.
\end{example}

\begin{definition}
Consider an application $r$ as above which gives, for each arrow $X \mr{f} Y$ of $\C$, another arrow $rX \mr{rf} rY$. Then, for any zigzag $P$ of arrows of $\C$, say \mbox{$X=X_0 \mr{f_0} X_1 \ml{f_1} X_2 \mr{f_2} ... X_n = Y$} 
we define the zigzag $rP$ applying $r$ pointwise: $rX=rX_0 \mr{rf_0} rX_1 \ml{rf_1} rX_2 \mr{rf_2} ... rX_n = rY$. 

\end{definition}

\begin{remark} \label{rem:rPenC0}
In the previous definition, note that:

\smallskip

\noindent -If $rf$ is a weak equivalence when $f$ is so, then when the backwards arrows of $P$ are weak equivalences, so are the ones of $rP$.

\smallskip

\noindent -If in addition all the arrows $rX \mr{rf} rY$ belong to a subcategory $\C_0$, then $rP$ is a zigzag of arrows of $\C_0$.
\end{remark}

\begin{remark} \label{rem:dosclasesdistintas}
Recall the definition of the homotopy relation  between arrows of $\C$, $f \sim_\cc{W} g$ if and only if $\gamma f = \gamma g$, we have also the relation $\sim_{\cc{W}_0}$ of homotopy in $\C_0$,  which relates two arrows $f,g$ if and only if $\gamma_0 f = \gamma_0 g$. 
Either by the definition of the equivalence relations in Definition \ref{def:construccionzigzags}, or by the universal properties involved, it is easy to see that when two arrows (or more generally two zigzags) of $\C_0$ are related by the equivalence relation defining $\HoCo$, then so are they 
by the one defining $\HoC$ when considered as zigzags of arrows of $\C$. We have thus
that $f \sim_{\cc{W}_0} g$ implies $f \sim_\cc{W}  g$, 
but the other implication doesn't necessarily hold.
Note that, in the hypothesis of previous remark, we can consider the class of $rP$ via the relation defining $\HoCo$ or $\HoC$, and this yields two arrows of $\HoC$ which are {\em a priori} different.
\end{remark}

We consider in what follows a left pointwise deformation $r$ of $\C$ into $\C_0$. As usual, there are dual versions for a right pontwise deformation which we omit.

\begin{lemma} \label{lema:zigzagreducidos}
Any zigzag $X \mr{P} Y$ of $\HoC$ is in the same class of the zigzag \mbox{$X \ml{\theta_X} rX \mr{rP} rY \mr{\theta_Y} Y$.}
\end{lemma}

\begin{proof}
From the diagram in Definition \ref{def:pointwisedeformation}, and the definition of the equivalence relation between zigzags in Definition \ref{def:construccionzigzags}, for each arrow $X \mr{f} Y$ it follows:

\smallskip

\noindent -The zigzag $X \mr{f} Y$ is equivalent to $X \ml{\theta_X} rX \mr{rf} rY \mr{\theta_Y} Y$.

\smallskip

\noindent -If $f$ is a weak equivalence, the zigzag $Y \ml{f} X$ is equivalent to $Y \ml{\theta_Y} rY \ml{rf} rX \mr{\theta_X} X$.

\smallskip

\noindent
From these two statements, the desired result follows immediately.
\end{proof}

\begin{definition} \label{def:HoCr}
We define the category $\HoCr$ as follows: its objects are those of $\C$ and an arrow $X \mr{} Y$ is given by the class in $\HoC$ of a zigzag of arrows of $\C_0$ from $rX$ to $rY$ in which the backward arrows are weak equivalences. Arrows are composed in $\HoC$, and the identities are given by those of $\HoC$.

We define a functor $\C \mr{\gamma_r} \HoCr$ which is the identity on objects and maps an arrow $f$ to the class of the zigzag 
of length one $rf$. 
All the verifications are straightforward.
\end{definition}

From Remark \ref{rem:rPenC0} and Lemma \ref{lema:zigzagreducidos} it follows 

\begin{corollary} \label{coro:HoCr}
We have a commutative diagram
$\vcenter{\xymatrix@R=1pc{ & \HoC \ar@<-1ex>[dd]_{\varphi} \\
\cc{C} \ar[ru]^{\gamma} \ar[rd]_{\gamma_r} \\
& \HoCr \ar@<-1ex>[uu]_{\psi}   }}$
in which $\varphi$ and $\psi$ are mutually inverse functors, given by ``conjugation with the arrows $\theta$": more precisely, both functors are the identity on objects and satisfy

\smallskip

\noindent 
$\varphi [X \mr{P} Y] = [rX \mr{\theta_X} X \mr{P} Y \ml{\theta_Y} rY] \quad ( \, = [rX \mr{rP} rY])$,

\smallskip

\noindent 
$\psi [rX \mr{Q} rY] = [X \ml{\theta_X} rX \mr{Q} rY \mr{\theta_Y} Y]$. 

In particular, $\C \mr{\gamma_r} \HoCr$ is the localization of $\C$ with respect to $\cc{W}$.
\qed
\end{corollary}

\begin{definition}
We say that $\C$ is (resp. pointwise) deformable into $\C_0$ it there is a finite sequence of (resp. pointwise) deformations \mbox{$\C \stackrel{r_1}{\curvearrowright}
\C_1 \stackrel{r_2}{\curvearrowright}
\C_2 ... \stackrel{r_n}{\curvearrowright}
\C_n = \C_0$.} In this case, for each arrow $X \mr{f} Y$ we denote by $rX \mr{rf} rY$ the arrow $r_n ... r_2 r_1 f$, and for each $X$ we denote by $\theta_X$ the zigzag from $rX$ to $X$ constructed from the $(\theta_i)_X$. 
\end{definition}

\begin{remark} \label{rem:rfwesiifwe}
As in Definition \ref{def:pointwisedeformation}, in the definition above we have by the two out of three property that $rf$ is a weak equivalence if and only if $f$ is so.
\end{remark}

We note that Lemma \ref{lema:zigzagreducidos}, Definition \ref{def:HoCr} and Corollary \ref{coro:HoCr} hold for a $\C$ which is pointwise deformable into $\C_0$ with the exact same formulations.

\begin{lemma} \label{lema:HoCrlengthone}
If $(\C_0,\cc{W})$ is Whitehead, then any arrow 
$X \mr{} Y$
of $\HoCr$ is represented by a zigzag of length one $rX \mr{} rY$.
\end{lemma}

\begin{proof}
By item 2 in Corollary \ref{coro:concreto2}, for each zigzag of arrows of $\C_0$ there is a zigzag of length one which is in the same class by the equivalence relation defining $\HoCo$ (see Remark \ref{rem:condition2}). Remark \ref{rem:dosclasesdistintas} finishes the proof.
\end{proof}

Combining Corollary \ref{coro:HoCr} and Lemma \ref{lema:HoCrlengthone}, we have that 
if $\C$ is pointwise deformable into a $\C_0$
which satisfies Whitehead, the functor $\C \mr{\gamma_r} \HoCr$ is the localization of $\C$ with respect to $\cc{W}$ and that each arrow 
$X \mr{} Y$ 
of $\HoCr$ is given by the class of a single arrow $rX \mr{} rY$ under the equivalence relation $\sim_\cc{W}$.
The reader should be aware that, in the case in which $\C$ is a model category, the relation $\sim_\cc{W}$ doesn't necessarily coincide with Quillen's notion of homotopy as we don't have the Whitehead condition for $\C$, only for $\C_{fc}$.   
That is the reason why this is as far as we can go, for categories with weak equivalences, with the notion of pointwise deformation. 
It is for showing that the relation $\sim_\cc{W}$ coincides with $\sim_{\cc{W}_0}$ when restricted to arrows of $\C_0$ that we will assume the functoriality of $r$, note that the following proposition follows from \cite[3.3,iv)]{DHKS} but we found it pertinent to give a proof.

\begin{proposition} \label{prop:rhoWyrhoW0}
If $r: \C \curvearrowright \C_0$ is a deformation, then for arrows $f,g$ of $\C_0$, 
$f \sim_\cc{W} g$ implies $f \sim_{\cc{W}_0} g$ and thus both relations coincide in $\C_0$.
\end{proposition}

\begin{proof}
Recalling Remark \ref{rem:rhoSigmasinlocalization},
Let $\C_0 \mr{F} \cc{D}$ be a functor which maps the weak equivalences to isomorphisms, 
and consider $\C \mr{Fr} \cc{D}$, thus if $f \sim_\cc{W} g$ we have $Frf = Frg$, and applying $F$ to the diagrams in Definition \ref{def:pointwisedeformation} it follows $Ff = Fg$.
\end{proof}

Combining Lemma \ref{lema:HoCrlengthone} and Proposition \ref{prop:rhoWyrhoW0} we have

\begin{proposition} \label{prop:final}
If $\C$ is deformable into a subcategory $\C_0$
which satisfies Whitehead, then 
the set $\HoCr(X,Y)$ is the set of homotopical classes $\C_0 / \sim_{\cc{W}_0} \!\ (rX,rY)$
\end{proposition}

Finally, combining
Proposition \ref{prop:splitiswhitehad}, 
 Corollary \ref{coro:HoCr} and Proposition \ref{prop:final} we have:
 
\begin{theorem} \label{theorem}
If $\C$ is deformable into a split-generated subcategory $\C_0$, then 
$\HoCr$ can be constructed with the same objects of $\C$ and with arrows from $X$ to $Y$ the homotopical classes in $\C_0$ from $rX$ to $rY$.
The functor 
$\C \mr{\gamma_r} \HoCr$ 
which is the identity on objects and maps an arrow $f$ to the class of $rf$ is 
the localization of $\C$ with respect to the weak equivalences. \qed
\end{theorem}

\begin{corollary} \label{coro:Csaturated}
If $(\C,\cc{W})$ is a homotopical category which is deformable into a split-generated subcategory $\C_0$ which satisfies the ``fork" condition, then $(\C,\cc{W})$ is saturated.
\end{corollary}

\begin{proof}
For an arrow $f$ of $\C$, by definition if $\gamma_r(f)$ is an isomorphism then so is $\gamma_0(rf)$, thus by Proposition \ref{prop:forkdasaturated} $rf$ is a weak equivalence and by Remark \ref{rem:rfwesiifwe} so is $f$.
\end{proof}

In the model category case we have the deformations $\C \stackrel{Q}{\curvearrowright} \C_c \stackrel{R}{\curvearrowright} \C_{fc}$ and $\C \stackrel{R}{\curvearrowright} \C_f \stackrel{Q}{\curvearrowright} \C_{fc}$
 (see \cite[10.3]{DHKS}), and thus by Example \ref{ex:retrsect} and by the results of \S \ref{sub:construction} and \S \ref{sub:modelcat},  we recover 
from Theorem \ref{theorem} and Corollary \ref{coro:Csaturated} 
 the classical results (recall that in $\C_{fc}$ all the considered notions of homotopy coincide):

\begin{theorem}
The localization of a model category $\C$ with respect to the weak equivalences can be constructed as the category with the same objects of $\C$ and with arrows from $X$ to $Y$ the homotopical classes in $\C_{fc}$ from $RQX$ to $RQY$. The localization functor is the identity on objects and maps an arrow $f$ to the class of $RQf$. \qed
\end{theorem}

\begin{corollary}
Any model category is saturated. \qed
\end{corollary}

\bibliographystyle{unsrt}

\begin{thebibliography}{99}

\bibitem{DHKS} 
Dwyer W.G., Hirschhorn P.S., Kan D.M., Smith J.H.,
\textsl{Homotopy Limit Functors on Model Categories
and Homotopical Categories}, AMS Mathematical Surveys and Monographs 113 (2004).

\bibitem{CartaGroth} 
Grothendieck A.,
\textsl{Lettre d'Alexandre Grothendieck sur les D\'erivateurs}, edited by M. K\"unzer, \url{http://webusers.imj-prg.fr/~georges.maltsiniotis/groth/ps/lettreder.pdf}
(1991).

\bibitem{GJ} Goerss P.J, Jardine J., \textsl{Simplicial Homotopy Theory}, Progress in Mathematics, Birkh\"auser (1996).

\bibitem{McL} Mac Lane S. \textsl{Categories for the Working Mathematician}, Graduate Texts in Mathematics Vol. 5 (1971).

\bibitem{Quillen} Quillen D., \textsl{Homotopical Algebra}, Springer Lecture Notes in Mathematics 43 (1967). 

\end{thebibliography}

\end{document}